\documentclass[a4paper,11pt]{amsart}

\usepackage{amsfonts,amssymb}
\usepackage{textcomp}

\theoremstyle{plain}

\newtheorem{thm}{Theorem}[section]
\newtheorem{dfn}[thm]{Definition}

\newtheorem{lem}[thm]{Lemma}

\newtheorem{rmk}[thm]{Remark}

\begin{document}

\title[Non-inner automorphisms of order $p$ in finite thin $p$-groups ]%
{The existence of non-inner automorphisms of order $p$ in finite thin $p$-groups }

\author[M.Ruscitti]{Marco Ruscitti}
\address{DISIM \\Universit\`a degli studi dell'Aquila\\ 67100 L'Aquila, Italy\\
{\it E-mail address}: { \tt marco.ruscitti@dm.univaq.it }}

\author[L.Legarreta]{Leire Legarreta}
\address{Matematika Saila\\ Euskal Herriko Unibertsitatea UPV/EHU\\
48080 Bilbao, Spain\\ {\it E-mail address}: {\tt leire.legarreta@ehu.eus}}

\thanks{The first author would like to thank the Department of Mathematics at the University of the Basque Country for its excellent hospitality while part of this paper was being written. The second author is supported by the Spanish Government, grants
MTM2011-28229-C02-02 and MTM2014-53810-C2-2-P, and by the Basque Government, grant IT753-13 and IT974-16.}

\keywords{Finite p-groups, non-inner automorphisms, derivation, thin p-groups\vspace{3pt}}

\begin{abstract}
In this paper we study the existence of at least one non-inner automorphism of order $p$ of a finite thin $p$-group, whenever the prime $p$ is.
\end{abstract}

\maketitle

\author{Marco Ruscitti}
\address{DISIM \\Universit\`a degli studi dell'Aquila\\ 67100 L'Aquila, Italy\\
{\it E-mail address}: { \tt marco.ruscitti@dm.univaq.it }}

\vspace{8pt}

\author{Leire Legarreta}
\address{Matematika Saila\\ Euskal Herriko Unibertsitatea UPV/EHU\\
48080 Bilbao, Spain\\ {\it E-mail address}: {\tt leire.legarreta@ehu.eus}}

\section{Introduction}

The main goal of this paper is to contribute to the longstanding conjecture of Berkovich posed in 1973, that conjectures that every finite $p$-group admits a non-inner automorphism of order $p$, where $p$ denotes a prime number \cite[Problem 4.13]{khukhro:2010}. The conjecture has attracted the attention of many mathematicians during the last couple of decades, and has been confirmed for many classes of finite $p$-groups. It is remarkable to put on record that, in 1965, Liebeck \cite{liebeck:1965} proved the existence of a non-inner automorphism of order $p$ in all finite $p$-groups of class $2$, where $p$ is an odd prime. However, the fact that there always exists a non-inner automorphism of order $2$ in all finite $2$-groups of class $2$ was proved by Abdollahi \cite{abdollahi:2007} in 2007. The conjecture was confirmed for finite regular $p$-groups by Schmid \cite{schmid:1980} in 1980. Indeed, Deaconescu \cite{deaconescu:2002} proved it for all finite $p$-groups $G$ which 
 are not 
 strongly Frattinian. Moreover, Abdollahi \cite{abdollahi:2010} proved it for finite $p$-groups $G$ such that $G/Z(G)$ is a powerful $p$-group, and Jamali and Visesh \cite{jamali:2013} did the same for finite $p$-groups with cyclic commutator subgroup. In the realm of finite groups, quite recently, the result has been confirmed  for semi-abelian $p$-groups by Benmoussa and  Guerboussa \cite{benmoussa:2015}, and for $p$-groups of nilpotency class $3$, by Abdollahi, Ghoraishi and Wilkens  \cite{abdollahi':2013}. Finally, Abdollahi et al \cite{abdollahi:2014} proved the conjecture for $p$-groups of coclass $2$.

With the contribution of this paper we add another class of finite $p$-groups to the above list, by proving that the above mentioned conjecture holds true for all finite thin $p$-groups, whenever the prime $p$ is.

The organization of the paper is as follows. In Section 2 we exhibit some preliminary facts and tools that will be used in the proof of the main result of the paper, and in Section 3 we recall elementary matters about thin $p$-groups and we prove the main result as well. 
\vspace{10pt}

\noindent
Throughout the paper, most of the notation is standard and it can be found, for instance, in \cite{Ro}.

\section{Preliminaries}
In this section, we recall some facts about derivations in the multiplicative setting, and some related lemmas, which will be useful to prove the main Theorem \ref{main} of the paper. The reader could be referred to  \cite{gavioli:1999} for more details and explicit proofs about derivations.

\begin{dfn}
Let $G$ be a group and let $M$ be a right $G$-module. A derivation $\delta:G \rightarrow M$ is a function such that $$\delta(gh)={\delta(g)}^h\delta(h),   \text{ for all  } g, h \in G.$$
\end{dfn}

In terms of its properties, it is well-known that a derivation is uniquely determined by its values over a set of generators of $G$. Let $F$ be a free group generated by a finite subset $X$ and let $G=\langle X : r_{1}, \ldots, r_{n} \rangle$ be a group whose free presentation is $F/R$, where $R$ is the normal closure of the set of relations $\{ r_{1} , \ldots , r_{n}\}$ of $G$. Then a standard argument shows that $M$ is a $G$-module if and only if $M$ is an $F$-module on which $R$ acts trivially. Indeed, if we denote by $\pi$ the canonical homomorphism $\pi: F \rightarrow G$, then the action of $F$ on $M$ is given by $mf=m\pi(f)$, for all $m\in M$ and all $f\in F$. 
Continuing with the same notation, we have the following results.

\begin{lem} \label{determineunique}
Let $M$ be an $F$-module. Then every function $f: X \rightarrow M$ extends in a unique way to a derivation $\delta: F \rightarrow M$.
\end{lem}

\begin{lem} \label{Der}
Let $M$ be a $G$-module and let $\delta: G \rightarrow M$ be a derivation. Then $\bar{\delta} : F \to M$ given by the composition $\overline{\delta}(f)=\delta(\pi(f))$ is a derivation  such that $\overline{\delta}(r_i)=0$ for all $i\in \{1,\ldots,n\}$. Conversely, if $\overline{\delta}:F \rightarrow M$ is a derivation such that $\overline{\delta}(r_i)=0$ for all $i \in \{1,\ldots, n\}$, then $\delta(fR)=\overline{\delta}(f)$ defines, uniquely, a derivation on $G=F/R$ to $M$ such that $\overline{\delta}=\delta \circ \pi$.
\end{lem}

In the following lemma, we study the relationship between derivations and automorphisms of a finite $p$-group. 

\begin{lem} \label{lift}
Let $G$ be a finite $p$-group and let $M$ be a normal abelian subgroup of $G$ viewed as  a $G$-module. Then for any derivation $\delta: G \rightarrow M$, we can define uniquely an endomorphism $\phi$ of $G$ such that $\phi(g)=g\delta (g)$ for all $g \in G$. Furthermore, if  $\delta(M)=1$,  then $\phi$ is  an automorphism of $G$.
\end{lem}

In order to reduce some calculations in terms of commutators, we keep in mind the following result.

\begin{lem} \label{free}
Let $F$ be a free group, $p$ be a prime number and $A$ be an $F$-module. If $\delta : F \rightarrow A$ is a derivation then, 
\begin{enumerate}
\item  $\delta(F^{p})=\delta(F)^{p}[\delta(F),_{p-1}F]$,
\item if $[A,_{i}F]=1$, we have $\delta(\gamma_{i}(F)) \leq [\delta(F), _{i-1}F]$ for all $ i \in \mathbb{N}$.
\end{enumerate}
\end{lem}

\begin{proof}
Let $ x \in F$. We have $\delta(x^{p})=\delta(x)^{x^{p-1} + x^{p-2} + \cdots + 1 }$. Since $(x-1)^{p-1} \equiv  x^{p-1} + x^{p-2} + \cdots + 1\mod p$, the first assertion follows. Now let us prove the second assertion by induction on $i$.
 Clearly, the assertion holds when $i=1$. 
Let us suppose, by inductive hypothesis that if $[A,_{k}F]=1$ 
then $\delta(\gamma_{k}(F)) \leq [\delta(F), _{k-1}F]$ for some $ k \in \mathbb{N}$. 
Let us take any $a \in F$ and any $b \in \gamma_{k}(F)$, and let us suppose 
that $[A,_{k+1}F]=1$. Then $$\delta([a,b])=[\delta(a),b][a,\delta(b)][a,b,\delta(a)][a,b,\delta(b)] \in 
  [\delta(F), _{k}F].$$
\end{proof}

\section{Berkovich Conjecture for finite thin $p$-groups}
To develop the section, we start enumerating some structural properties concerned about thin $p$-groups. Firstly, let us recall that in a group $G$ an antichain is a set of mutually incomparable elements in the lattice of its normal subgroups.  It is well-known that, if $G$ is a $p$-group of maximal class, then the lattice of its normal subgroups consists of $p+1$ maximal subgroups and of the terms of the lower central series of $G$. Thus, a $p$-group of maximal class has only one antichain, which consists of its maximal subgroups. The necessity to extend the family of groups of maximal class to a bigger family of $p$-groups with a bound on the antichains, leads us to introduce the formal definition of thin $p$-group. Let us introduce the definition of thin $p$-groups as in \cite{caranti:1996}. Forwards, we mention as well some already proved results about the existence of non-inner automorphisms of order $p$ in certain specific cases to avoid repetitions.

\begin{dfn}
Let $G$ be a finite $p$-group. Then $G$ is thin if every antichain in $G$ contains at most $p+1$ elements.
\end{dfn}
The following results about finite thin $p$-groups are discussed in \cite{brandl:1992}.
\begin{lem} \label{basics}
Let $G$ be a finite thin $p$-group, and let $p$ be an odd prime. 
\begin{enumerate}
\item If $N$ is a normal subgroup of $G$,  then $N$ is a term of the lower central series of $G$ if and only if $N$ is the unique normal subgroup of its order.
\item If $G$ is not of maximal class, then $|G| \geq p^{5}$ and $G/ \gamma_{3}(G)$ is of exponent $p$.
\end{enumerate}
\end{lem}

\begin{lem} \label{cp}
Let $G$ be a finite thin $p$-group and let us take any $h \in \gamma_{i}(G) - \gamma_{i+1}(G)$. Then the following called coverty property holds, $$[h,G]\gamma_{i+2}(G)=\gamma_{i+1}(G).$$
\end{lem}

\begin{rmk} \label{structurethin}
The properties shown in Lemma \ref{basics} give us a lot of information about thin $p$-groups. Since an elementary abelian $p$-group is thin if and only if  its order  is $p^{2}$ (see \cite{caranti:1996}), then every finite thin $p$-group is a two generator group. Secondly, by \cite{brandl:1992} we know that the lower and the upper central series of this kind of groups coincide, and consequently that all quotients of these series are elementary abelian $p$-groups of order at most $p^{2}$. Consequently, $Z(G)$ must be cyclic of order $p$. In particular, the quotients of the lower and upper central of finite thin $p$-groups have exponent $p$.  Moreover, if $G$ is a finite thin $p$-group, then $\Phi(G)=\gamma_{2}(G)$, $Z_2(G)\leq Z(\Phi(G))$ and the property stated in the above Lemma \ref{cp} holds equivalently for terms of the upper central series of $G$. 
\end{rmk}

Since Liebeck in \cite{liebeck:1965} proved the existence of at least a non-inner automorphism of order $p$ in all finite $p$-groups of class $2$, where $p$ is an odd prime,  Abdollahi in \cite{abdollahi:2007} proved  the existence of such an non-inner automorphism of order $2$ in all finite $2$-groups of class $2$,  and Abdollahi, Ghoraishi and Wilkens in  \cite{abdollahi':2013} did the same in the case of finite $p$-groups of nilpotency class $3$, in our study, we can deal with finite $p$-groups of nilpotency class $c\geq 4$. On the other hand,  since Deaconescu in \cite{deaconescu:2002} proved the existence of at least a non-inner automorphism of order $p$ for all finite $p$-groups $G$ which are not strongly Frattinian, we may assume that the finite thin $p$-groups $G$ we are interested in, are strongly Frattinian, in other words, that the groups of our interest satisfy  $C_{G}(\Phi(G))=Z(\Phi(G))$. As a result due to Abdollahi in \cite{abdollahi:2010}, we have that if $G$
  is a fi
 nite $p$-group such that $G$ has no non-inner automorphisms of order $p$ leaving $\Phi(G)$ elementwise fixed, then $d(Z_{2}(G)/Z(G))=d(G)d(Z(G))$. Thus, in view of this previous result, in order to prove the existence of a non-inner automorphism of order $p$ in a finite thin $p$-group $G$, we may assume without any loss of generality that the condition $d(Z_{2}(G)/Z(G)) = d(G)d(Z(G))$ holds. Thus, using the consequences got in Remark \ref{structurethin} we have $Z_{2}(G)/Z(G)$ is isomorphic to an elementary abelian group of order $p^2$.

It is also showed in \cite{brandl:1988} (see Theorem $B$) that every finite thin $2$-group is a group of maximal class. Thus in the following we focus our attention on finite thin $p$-groups, where $p$ is an odd prime.

Now we are ready to prove the next theorem.

\begin{thm} \label{main}
Let $G$ be a finite thin $p$-group, where  $p$ is an odd prime. Then $G$ has a non-inner automorphism of order $p$.
\end{thm}
\begin{proof}
Let  us denote $c$ the nilpotency class of $G$. As it has been said before, we can assume $c\geq 4$. From Remark  \ref{structurethin} and its consequences and assumptions we know that $G$ is a two generator $p$-group, the lower and the upper central series of $G$ coincide,  $\Phi(G)=\gamma_{2}(G)$, $Z_2(G)\leq Z(\Phi(G))$, $Z(G) \cong C_{p}$, $d(Z_{2}(G)/Z(G)))=d(G)d(Z(G))$ and  $Z_{2}(G)/Z(G) \cong C_{p} \times C_{p}$. In particular, $[Z_{2}(G), \gamma_{2}(G)]=1$ and $\Omega_{1}(Z_{2}(G))$ is an elementary abelian subgroup of $G$. 

Indeed, $G/\gamma_{3}(G)$ has order $p^{3}$, class $2$ and  by Lemma \ref{basics} we may assume that its exponent is $p$, i.e. $G/ \gamma_{3}(G)$ is an extraspecial $p$-group of exponent $p$ and order $p^{3}$.   

Our goal is to obtain at least an automorphism of $G$ of order $p$. Firstly, we  define an assignment  on generators of the free group generated by two elements sending them to  $\Omega_{1}(Z_{2}(G))$. By Lemma \ref{determineunique} it is possible to extend these assignments to a derivation. Secondly, we show that this map preserves the relations defining the quotient $G/ \gamma_{3}(G)$, and then we apply Lemma \ref{Der} to induce a derivation from  $G/ \gamma_{3}(G)$ to $\Omega_{1}(Z_{2}(G))$. Finally, we lift this found map to a derivation from $G$ to $\Omega_{1}(Z_{2}(G))$, applying Lemma \ref{lift}. In the following paragraphs we describe in detail each of the mentioned steps.
\noindent

Let $x \rightarrow u, y \rightarrow v $ be an assignment on generators $x,y$ of the two generator free group $F_{2}$, with $u,v \in \Omega_{1}(Z_{2}(G))$. By Lemma \ref{Der} this assignment extends uniquely to a derivation $ \delta : F_{2} \rightarrow \Omega_{1}(Z_{2}(G))$ such that $\delta(x)=u$ and $\delta(y)=v$. Making perhaps some abuse of notation, we can assume that $G/ \gamma_{3}(G)$ corresponds to the presentation  $ \langle x,y  \ | \ x^{p}, y^{p}, [y,x,x], [y,x,y] \rangle$. Now let us see that the equalities $\delta(x^{p})=1$, $\delta(y^{p})=1$, $\delta([y,x,x])=1$ and $\delta([y,x,y])=1$ hold. To do it, firstly, let us show that $\delta$ is trivial on the $p$-th powers of elements of $F_2$. In fact, considering $\pi'$ the canonical epimorphism from $F_{2}$ to $G/\gamma_{3}(G)$ we have
$$\delta(f^{p})= \delta(f)^{f^{p-1}}\delta(f^{p-1})= \cdots = \delta(f)^{f^{p-1}+ \cdots + 1}=\delta(f)^{\pi'(f^{p-1}+ \cdots + 1)}=$$
$$= \delta(f)^{\pi'(f)^{p-1}+ \cdots + 1}=\delta(f)^{p}[\delta(f),\pi'(f)]^{\binom{p}{2}}=1, \quad{ for \ all \ } f \in F_2.$$
Secondly, let us analyze the behaviour of $\delta$ on commutators. Since for any $g,h \in G$ it holds that $gh=hg[g,h]$, then applying $\delta$ to this previous equality, we get  $\delta(gh)=\delta(hg[g,h])=\delta(hg)^{[g,h]}\delta([g,h])$, and consequently taking also into account that $[Z_{2}(G), \gamma_{2}(G)]=1$, we get
$\delta([g,h])=\delta(gh) (\delta(hg)^{[g,h]})^{-1}=\delta(gh) (\delta(hg))^{-1}=[\delta(g),h][g,\delta(h)]$, which is an element of $Z(G)$.

Moreover,
$$  \delta([y,x,y])=[\delta([y,x]),y] [[y,x],\delta(y)]=1,$$    $$\delta([y,x,x])=[\delta([y,x]),x] [[y,x],\delta(x)]=1.$$
(Take into account that the previous two equalities can be obtained as well, applying properly item (ii) of Lemma \ref{free}.)

\noindent

By Lemma \ref{Der} we can induce a derivation $\overline{\delta}$ from  $G/ \gamma_{3}(G)$ to $\Omega_{1}(Z_{2}(G))$. The map $\delta' :G \rightarrow \Omega_{1}(Z_{2}(G))$ defined by the law $\delta ' (g)=\overline{\delta}(g\gamma_{3}(G))$ for all $ g \in G$, is a derivation from $G$ to $ \Omega_{1}(Z_{2}(G))$.  By Lemma \ref{lift}, $\delta'$ induces an automorphism $\phi$ of $G$ by the law $\phi(g)=g\delta'(g)$ for all $g \in G$, leaving, in particular, $\Phi(G)$ elementwise fixed. Clearly, $Z_{2}(G) \leq Z(\gamma_{3}(G))$. In fact, the previous inclusion holds since $Z_{2}(G) \leq Z(\Phi(G))$, $\Phi(G)=\gamma_{2}(G)$ and $Z_{2}(G) = \gamma_{c-1}(G) \leq \gamma_{3}(G)$. This allows us to prove that the automorphism $\phi$ has order $p$. In fact,  
$\phi^{p}(g)=\phi^{p-1}(g\delta'(g))=\phi^{p-1}(g) \phi^{p-1}(\delta'(g))=\phi^{p-1}(g) \delta'(g)=\phi^{p-2}(g\delta'(g))\delta'(g)=\cdots= g (\delta'(g))^{p}=g$, for all $g \in G$.

\noindent

Thus, doing this previous construction we produce a set of automorphisms of $G$ of order $p$, whose size is equal to $|\Omega_{1}(Z_{2}(G))|^{2}$, in other words, whose size is the number of possible choices for the images of the above generators. Now, we distinguish the only two possible cases: $Z_{2}(G) \cong C_{p} \times C_{p} \times C_{p}$ or  $Z_{2}(G) \cong C_{p^{2}} \times C_{p}$. In the former case, when  $Z_{2}(G) \cong C_{p} \times C_{p} \times C_{p}$, we produce $p^{6}$ automorphisms of $G$ of order $p$. However, the number of inner automorphisms of $G$ induced by elements of $Z_{3}(G)$ is at most $p^{4}$.  Thus, a counting argument is enough to say that $G$ has a non-inner automorphism of $G$ of order $p$, and we get the statement of the Theorem in this case. Otherwise, in the latter case, when $Z_{2}(G) \cong C_{p^{2}} \times C_{p}$, choose an assignment $x \rightarrow u, y \rightarrow v $ such that $u,v \in \Omega_{1}(Z_{2}(G)) - \{1\}$, $u$ is not central, and
  let $\phi$ be the automorphism of $G$ of order $p$ obtained by this assignment. On the other hand, we know that $\phi$ is inner if and only if there exists an element $h_{\phi}$ of $Z_{3}(G) - Z_{2}(G)$ such that $h_{\phi}^{p} \in Z(G)$, $\phi(g)=g^{h_{\phi}}$ and $\delta'(g)=[g,h_{\phi}]$, for all $g \in G$. Under these circumstances, since $Z(G) \leq \Omega_{1}(Z_{2}(G))$ we would deduce that  $[h_{\phi}, G]Z(G) \leq \Omega_{1}(Z_{2}(G))  < Z_{2}(G)$,  which is in contradiction with the analogous coverty property of Lemma \ref{cp}. Consequently, this second case does not happen and the statement of the Theorem is proved.
\end{proof}

\section{Acknowledgements}
The first author would like to thank the Department of Mathematics at the University of the Basque Country for its excellent hospitality while part of this paper was being written; he also wish to thank Professors Gustavo A. Fern\'andez Alcober, Norberto Gavioli and Carlo Maria Scoppola for their suggestions.

\bibliography{articles}

\bigskip
\bigskip

\end{document}